\documentclass{amsart}

\title[Periodic points on SFTs as a commensurability invariant]{Periodic points on shifts of finite type and commensurability invariants of groups}

\author{David Carroll}
\address{Dept. Of Mathematics, Mailstop 3368, Texas A\&M University, College Station, TX 77843-3368, USA}
\email{carroll@math.tamu.edu}

\author{Andrew Penland}
\address{Dept. of Mathematics and Computer Science, Western Carolina
  University, Stillwell 426, Cullowhee, NC 28723, USA}
\email{adpenland@email.wcu.edu}



\DeclareMathOperator{\supp}{supp}

\newcommand\St{\text{Stab}}
\newcommand\Fi{\text{Fix}}
\newcommand\Z{\mathbb{Z}}

\newcommand\Om{\Omega}
\renewcommand\S{\psi_{H,T}}
\newcommand\ov{\overline}
\newcommand\Oom{\overline{\Omega}}

\theoremstyle{plain}
\newtheorem{theorem}{Theorem}

\newtheorem{proposition}[theorem]{Proposition}
\newtheorem{cor}[theorem]{Corollary}
\newtheorem*{conjecture}{Conjecture}
\newtheorem*{question}{Question}

\theoremstyle{definition}
\newtheorem*{definition}{Definition}
\newtheorem*{remark}{Remark}

\numberwithin{equation}{section}

\begin{document} 
 
\begin{abstract}
We explore the relationship between subgroups and the possible shifts of finite type (SFTs) which can be defined on a group. In particular, we investigate two group invariants, weak periodicity and strong periodicity, defined via symbolic dynamics on the group. We show that these properties are invariants of commensurability.
Thus, many known results about periodic points in SFTs defined over groups are actually results about entire commensurability classes.
Additionally, we show that the property of being not strongly periodic (i.e. the property of
having a weakly aperiodic SFT) is preserved under extensions
with finitely generated kernels.
We conclude by raising questions and conjectures about the relationship of these invariants to the geometric notions of quasi-isometry and growth. 
\end{abstract} 
\maketitle

\section{Introduction}

Given a group $G$ and a finite alphabet $A$, a \emph{shift} (over $G$) is a non-empty compact, $G$-equivariant subset of the \emph{full shift} $A^G$, where $A^G$ is a topological space regarded as a product of discrete spaces.
A \emph{shift of finite type} (SFT) $S$ over $G$ is a shift defined by specifying a finite set of patterns
that are forbidden from appearing in elements of $S$.

 It is natural to ask how properties of the group are related to the dynamical properties of its SFTs. Since it is possible to define essentially equivalent SFTs on a group using different alphabets and different sets of forbidden patterns, we should seek properties which are defined in terms of \emph{all} possible SFTs defined on the group. 

One of the most natural properties of a dynamical system is the existence and type of its periodic points. We say that $G$ is \textit{weakly periodic} if every SFT defined on $G$ has a periodic point, and we say $G$ is \emph{strongly periodic} if every SFT defined on $G$ has a periodic point with finite orbit.
It should be noted that \emph{weakly periodic} is precisely the negation of what many authors refer to as
``the existence of a strongly aperiodic SFT.''

When the shifts in question are not required to be of finite type, the
question of periodic points is settled. Gao, Jackson, and Seward
proved that every countable group has a shift with no periodic
points~\cite{GaoJacksonSeward}. Previously, Dranishnikov and Schroeder had proved
that having a shift without strongly periodic points is a
commensurability invariant~\cite{DranishnikovSchroeder}. They also constructed
strongly aperiodic shifts on torsion-free hyperbolic groups and Coxeter groups.

Our objective is to show that for a group $G$, the dynamical
properties of SFTs on $G$ are influenced by its finite index subgroups.  We prove the following results.

\begin{theorem}\label{t:rigidly-periodic}
Let $G_1$ and $G_2$ be finitely generated
commensurable groups.  If $G_1$ is weakly periodic,
then $G_2$ is weakly periodic.
\end{theorem}

\begin{theorem}\label{t:strongly-periodic}
Let $G_1$ and $G_2$ be finitely generated commensurable groups. If $G_1$ is strongly periodic, then $G_2$ is strongly periodic. 
\end{theorem}

\begin{theorem}\label{t:extension}
Suppose $\displaystyle{1 \rightarrow N \rightarrow G \rightarrow Q \rightarrow 1}$ is
a short exact sequence of groups and $N$ is finitely generated.  If $G$ is strongly
periodic, then $Q$ is strongly periodic.
\end{theorem}

From Theorem \ref{t:extension} we obtain a proof of the following result, which can also
be inferred from the work of Ballier and Stein in \cite[Theorem 1.2]{BallierStein}.

\begin{cor}
If $G$ is a finitely generated group of polynomial growth $\gamma(n) \sim n^d$, where $d \geq 2$,
then $G$ is not strongly periodic.
\end{cor}

The study of periodic points in SFTs over arbitrary groups is well-established and arises from two major sources. The first is a generalization of classical symbolic dynamics, which considers shift spaces defined on $\mathbb{Z}$ or $\mathbb{Z}^2$, to arbitrary semigroups~\cite{SymbolicHyperbolic} and groups~\cite{CAGroups}. For instance, the group $\mathbb{Z}$ is a strongly periodic group, which follows from the fact (well-known in symbolic dynamics) that every two-sided SFT contains a periodic point and the fact (well-known in group theory) that every subgroup of $\mathbb{Z}$ has finite index. The proof is a straightforward exercise in the use of the well-known \textit{higher block shift}, a notion which we will generalize to arbitrary groups in Section \ref{HigherBlock}.
It is worth noting that the higher block shift for arbitrary groups is implicitly suggested in
~\cite{FreeSymbolicThesis} and ~\cite{CAGroups}.
It follows from work of Piantadosi (\cite{FreeSymbolicThesis}, \cite{FreeSymbolicArticle}) that a free group on $n > 1$ generators is weakly periodic but not strongly periodic.  

A second motivation for studying SFTs over arbitrary groups is
connected to tilings on Riemannian manifolds. Many authors have been
interested in sets of tiles which can tile (tesselate) a manifold
without admitting any translation (i.e. fixed-point-free) symmetries
from the manifold's isometry group. In certain cases, such tilings
naturally lead to shifts on a group. The question of the existence of
a set of tiles which can tesselate the plane $\mathbb{R}^2$ without
any translation symmetries was first asked by Wang~\cite{Wang}, who
was motivated by connections to decidability
problems. Berger~\cite{DominoProblem} constructed such a set of tiles,
now known as \emph{Wang tiles}. From this work, it follows that
$\mathbb{Z}^2$ is neither weakly periodic nor strongly periodic. The
generalization of this result for free abelian groups of rank $n > 2$
is due to Culik and Kari ~\cite{NWangTiles}.  Block and
Weinberger~\cite{BlockWeinberger} used a homology theory connected to
coarse geometry to construct aperiodic tiling systems for a large
class of metric spaces, including all nonamenable manifolds. It follows from their results that nonamenable
groups are not strongly periodic. Mozes~\cite{Mozes} constructed
aperiodic tiling systems for a class of Lie groups which contain
uniform lattices satisfying certain conditions. Mozes also associated
these tiling systems to labelings of vertices of the Cayley graph of a
lattice.  Using a homology theory related to the one used by Block and
Weinberger,
Marcinkowski and Nowak proved that a large class of amenable groups, namely those for which all of maximal subgroups have index not divisible by some prime $p$,
are not strongly periodic~\cite{MarcinkowskiNowak}. This class includes the Grigorchuk group. 

It should be noted that \emph{weakly periodic} is precisely the opposite of what some authors call \emph{strong
aperiodicity} or \emph{the
existence of a strongly aperiodic SFT}, i.e. a SFT on $G$ which has no
$G$-periodic points. Recently, Cohen~\cite{Cohen} has independently
obtained results connecting strong aperiodicity to the coarse geometry
of groups. In particular, he has shown that in the case of finitely
presented torsion-free groups, strong aperiodicity is a quasi-isometry
invariant, and that no group with at least two ends is strongly
aperiodic. Jeandel~\cite{Jeandel} showed that any finitely presented
group with a strongly aperiodic SFT must have decidable word problem,
disproving a conjecture of Cohen's that every one-ended group has a
strongly aperiodic SFT. In addition, Jeandel constructed strongly
aperiodic SFTs on a large class of groups. Cohen also conjectured that
strong aperiodicity is a quasi-isometric invariant for all finitely
generated groups, and this conjecture is still open. Since
commensurability implies quasi-isometry (but not vice versa), a proof
of his conjecture would subsume Theorem ~\ref{t:rigidly-periodic}.
The techniques we use in the paper are not explicitly
geometric. However, Cohen's results, as well as the connection in
other work between SFTs and isometry groups of Riemannian manifolds, suggest that a deeper geometric investigation of SFTs on groups is warranted. 

\textbf{Acknowledgments.} We are indebted to many people without whom this work would not have been possible. This work was motivated by our experiences at the 2014 Rice Dynamics Seminar, and so we thank the organizers and attendees of that conference. In particular, we first became interested of the potential connections between SFTs and group invariants through a talk by Ay\c{s}e \c{S}ahin and discussions with David Cohen.
We also thank Emmanuel Jeandel for providing useful feedback and 
our advisor Zoran \v{S}uni\'{c} for his valuable guidance.
Finally, we thank the anonymous referee for various helpful suggestions,
including several references.


\section{Background}\label{s:BG}

\subsection{Symbolic dynamics on groups}

Throughout this section, we assume $A$ is a finite discrete set,
called the \emph{alphabet},
and $G$ is an arbitrary group.
The set of all functions $G \rightarrow A$ is called \emph{the
full shift of $A$ over $G$} and is denoted $A^G$.
We endow $A^G$ with the product topology and
the (continuous) left $G$-action defined by $(gx)(h) = x(g^{-1} h)$ for
$g, h^ \in G$, $x \in A^G$.  Elements of $A^G$ are called \emph{configurations}.  Often we just call $A^G$ \emph{the full
shift} if $A$ and $G$ are understood.

A subset $X \subset A^G$ is \emph{shift-invariant} if $gx \in X$
for all $g \in G$, $x \in X$.
The stabilizer group of a configuration $x$ will
be denoted $\St_G(x) = \St(x) = \{ g \in G: gx = x\}$, and we call $x \in A^G$
\emph{weakly periodic}, or just
\emph{periodic}, if $\St(x)$ is nontrivial.
If $\St(x)$ is of finite index in $G$, so that $x$ has finite $G$-orbit, we call $x$ \emph{strongly periodic}.

\begin{definition}
$S \subset A^G$ is a \emph{shift space (or shift) of $A$ over} $G$
if $S$ is closed and shift-invariant.
\end{definition}

An equivalent definition can be formulated as follows.  Call a function
$p : \Omega \rightarrow A$ a \emph{pattern} if $\Omega$ is a finite subset of
$G$ and write $\supp(p) = \Omega$.
We say $p$ \emph{appears in} a configuration $x \in A^G$ if
$(gx)\big|_{\Omega} = p$ for some $g \in G$.  Then we have:

\begin{proposition}
$S \subset A^G$ is a shift if and only if there exists a set of patterns
$\displaystyle{P \subset \bigcup_{\substack{\Omega \subset G \\ \Omega
\text{ finite}}} A^\Omega}$
such that
$S = \{x \in A^G : \text{ no } p \in P \text{ appears in } x\}$.
\end{proposition}

$P$ as above is called a \emph{set of forbidden patterns} for $S$.
If $P$ can be taken to be finite, we call $S$ a \emph{shift of finite type}. 

\begin{remark}
Let $S$ be a SFT and $P$ be a finite set of forbidden patterns defining $S$.
We may assume that $P \subset A^\Omega$ for a fixed finite subset $\Omega \subset G$
by taking
$\Omega \supset \bigcup_{p \in P} \supp(p)$
and extending forbidden patterns in all possible ways.
In fact, if $G$ is finitely generated,
we may assume that $\Omega$ is a ball of radius $n$ with respect to some finite generating set of $G$. 
\end{remark}

Henceforth, we use the acronym SFT
as shorthand for a \emph{nonempty} shift of finite type of some finite alphabet
over a group. 

Different shift spaces over a group $G$ may use use differing alphabets.
Our primary concern is for properties common to all shift spaces over a
fixed group $G$, independent of the choice of finite alphabet.  We begin with
two such invariants.

\begin{definition}
$G$ is \emph{weakly periodic} if every SFT
over $G$ contains a (weakly) periodic configuration.
\end{definition}

\begin{definition}
$G$ is \emph{strongly periodic} if every SFT over $G$ contains a strongly periodic configuration.
\end{definition}
In other words, $G$ is strongly periodic if every SFT over $G$ contains a point with finite
orbit under the $G$-shift action.
Currently the only known strongly periodic groups of which the authors are aware are virtually cyclic groups.

As mentioned in the introduction, the properties which we call weak periodicity and strong periodicity have been established for several well-known classes of groups. We mention that it is not hard to see that any finite group is strongly periodic but not weakly periodic. For infinite groups, strongly periodic implies weakly periodic. 

\subsection{Quasi-isometry and commensurability of groups}

Here we give a very brief introduction to certain geometric notions in group theory, nearly all of which can be found in the textbook ~\cite{GGT}. 

For the remainder of this subsection, let $G$ be a group with a finite generating set $A = \{a_1, a_2, \ldots, a_m \}$, so that any group element in $G$ can be represented by at least one word in $A^{\pm 1}$.

We define the \emph{length} of $g \in G$ to be the length of the shortest word in $A^{\pm 1}$ which represents $g$. The length of $g$ is denoted by $|g|_A$, or $|g|$ if $A$ is understood.  This makes $G$ a metric space with
left-invariant
distance function $d_G^A$ given by $d_G^A(g_1,g_2) = |g_1^{-1}g_2|_A$ for $g_1, g_2 \in G$. This distance is equivalent to the combinatorial distance between $g_1$ and $g_2$ in the right Cayley graph of $G$ with respect to $A$. 

The \emph{growth function} of $G$ (with respect to $A$) is a map $\gamma_G^{A}: \mathbb{N} \rightarrow \mathbb{N}$ which counts the number of elements in a ball of radius $n$ with respect to $d_G^A$, i.e.  \[ \gamma_G^{A}(n) = | \{g \in G \; \mid \; |g| \leq n \} |. \] The growth function $\gamma_G^{A}$ depends on $A$, but this dependence is superficial. Take a partial order $\prec$ on growth functions so that $f \prec g$ if and only if there exists a constant $C$ so that $f(n) \leq g(C n)$ for all $n \in \mathbb{N}$, and define an equivalence relation on this set of functions so that $\simeq$ where $f \simeq g$ if and only if $f \prec g$ and $g \prec f$.  We call an equivalence class under $\simeq$ a \emph{growth type}. It is not hard to see that the growth type of a finitely generated group is an invariant of the group (independent of the choice of finite generating set). 

 A striking connection between the growth of a group and algebraic properties of its finite index subgroups is due to Gromov~\cite{PolynomialGrowth}. Recall that if $P$ is a property of groups, a group $G$ is \textit{virtually P} if $G$ has a finite index subgroup which has the property $P$. If $G$ and $H$ are groups, we say that $G$ is \textit{virtually H} if $G$ has a finite index subgroup isomorphic to $H$.

\begin{theorem}[Gromov]
A finitely generated group has polynomial growth if and only if it is virtually nilpotent. 
\end{theorem}

Growth is related to the coarse geometry of a group, as expressed by the notion of \emph{quasi-isometry}. 

\begin{definition} 
Let $X$ and $Y$ be metric spaces. A map $\phi: X \rightarrow Y$ is a \emph{quasi-isometry} if it satisfies the following properties:
\begin{itemize}
\item[(i)] there exist constants $B \geq 1$, $C \geq 0$ such that for all $x_1,x_2 \in X$, \[ \frac{1}{B}d_X(x_1,x_2) - C \leq d_Y(\phi(x_1),\phi(x_2)) \leq B d_X(x_1,x_2) + C. \]
\item[(ii)] there exists $D \geq 0$ such that for any $y \in Y$, there exists $x \in X$ with $d_Y(\phi(x),y) \leq D$. 
\end{itemize}
\end{definition}

Quasi-isometry gives an equivalence relation on metric spaces which captures the large-scale aspects of their geometric structures. Intuitively, two metric spaces are quasi-isometric if they are hard to tell apart from far away. It is not hard to see that all bounded metric spaces are quasi-isometric to one another, and that $\mathbb{Z}^n$ is quasi-isometric to $\mathbb{R}^n$ for any $n \geq 1$. 

Two groups $G_1,G_2$ are said to be \textit{commensurable} if there exists a group $H$ such that $G_1$ and $G_2$ are both virtually $H$. Commensurability is an equivalence relation on the class of all groups. Commensurability implies quasi-isometry, but the converse is not true (see ~\cite[IV.44, IV.47, and IV.48]{GGT}).

\section{Some special SFTs on groups}

In this section we construct some useful shifts of finite type
and apply them to the question of shift periodicity of groups.

\subsection{Higher block shift}\label{HigherBlock} Let $A$ be a finite set. Let $G$ be a group, $H$ a subgroup of finite index,
and $T$ a finite set of elements of $G$ such that $HT = G$.  Set $B = A^T$.
The \emph{higher block map} $\S : A^G \rightarrow B^H$ is defined by
$(\S(x))(h) = (h^{-1} x)\big|_T$.  In other words, $\S(x)$ is the
configuration $z \in B^H$ such that
\begin{equation} \label{hbFormula}
z(h)(t) = x(ht) \text{ for any } h \in H, \; t \in T.
\end{equation}
Notice that \eqref{hbFormula} guarantees $\S$ is injective, as $HT=G$.  However, not all configurations
$z$
in $B^H$ are of the form \eqref{hbFormula} for some $x \in A^G$---one
needs to restrict to those which overlap correctly.

We define a set of \emph{non-overlapping patterns}
$P$ for $B^H$ as follows.  Let $E = H \cap TT^{-1}$ and
\begin{align*}
P = \{ p \in B^E : &\; \text{there exist } h \in E 
\text{ and } t \in T \text{ such that } h^{-1} t \in T \\
&\; \text{and }p(1)(t) \neq p(h)(h^{-1} t)\}.
\end{align*}
The \emph{higher block shift} (over $G$, relative to
$H,T$) is the shift of finite type
$I \subset B^H$ defined by the forbidden patterns $P$.
\begin{proposition} \label{HigherBlockBij}
$\psi_{H,T}$ maps $A^G$ onto $I$ bijectively.
\end{proposition}

\begin{proof}
First we show $z = \S(x) \in I$ for any $x \in A^G$, which amounts to showing that for any $k \in H$,
the restriction $r = (kz)\big|_E$ is not in $P$.  For $h \in E$ and $t \in T$ such that
$h^{-1} t \in T$,
\[
r(1)(t) =(kz(1))(t)= z(k^{-1})(t) = x(k^{-1} h h^{-1} t) = z(k^{-1} h)(h^{-1} t) = r(h)(h^{-1} t).
\]
Hence $r \notin P$.  

It remains to show $\S $ is surjective.
If $z \in I$, define $x \in A^G$ by setting $x(ht) = z(h)(t)$ for $h \in H$, $t \in T$.
To see $x$ is well-defined, suppose
$h_2 \in H$, $t_2 \in T$ are such that $ht = h_2 t_2$.  Then $h^{-1} h_2 = t t_2^{-1} \in E$ and
$h_2^{-1} h t = t_2 \in T$.  Thus the non-overlapping restrictions in $I$ imply that
\[
z(h)(t) = (h^{-1} z)(1)(t) = ((h^{-1} z)(h^{-1} h_2))(h_2^{-1} h t) = z(h_2)(t_2).
\]
We conclude $x$ is well-defined, and it is evident that $\S(x) = z$.
\end{proof}

\begin{remark} \label{StabNest}
Notice that if $z \in I$ and $x = \S^{-1}(z)$, $\St_H(z)$ is a subgroup of $\St_G(x)$.
\end{remark}

\subsection{Products of shifts} \label{s:prod}

In this section we recall some basic results about products of shift spaces.
Let $A$ and $B$ be finite alphabets.  We set $C = A \times B$.
If $c = (a,b) \in A \times B$,
we use the notation $c_1 = a$, $c_2 = b$.

Suppose now that $X$ is an arbitrary set.
Notice that $C^X$
can be identified with $A^X \times B^X$:
if $f: X \rightarrow C$
is a function, we write $f_1 : X \rightarrow C \rightarrow A$ and
$f_2: X \rightarrow C \rightarrow B$ for $f$ composed with the corresponding
projections, and make the identification $f = (f_1, f_2)$.

Some easy consequences of this identification are as follows.
If $G$ is a group and $x \in C^G$, $(gx)_i = g(x_i)$ for $i = 1,2$.
If $\Omega \subset G$, then $(x|_\Omega)_i = (x_i)|_\Omega$ for $i = 1,2$.
If $S_1 \subset A^G$, $S_2 \subset B^G$ are shifts on
$A$ and $B$, respectively, then $S_1 \times S_2$ is a shift on $C$.
Moreover, if $S_1$ is a SFT defined by forbidden patterns $P \subset A^\Omega$,
then the SFT on $C$ defined by forbidden patterns $P \times B^\Omega$
equals $S_1 \times B^G$.  It follows that if $S_1$ and $S_2$ are
SFTs,
then $S_1 \times S_2$ is a SFT:
assuming $S_1$ is defined by forbidden patterns $P \subset A^\Omega$
and $S_2$ by forbidden patterns $Q \subset B^\Omega$,
$S_1 \times S_2$ is the SFT defined by forbidden patterns
\[ P \times B^\Omega \; \cup \; A^\Omega \times Q. \]
This result is given in \cite[Proposition 3.4]{SymbolicHyperbolic} with a different argument.

\subsection{Locked shift} \label{LockedShift}

Suppose $A$ is a finite alphabet,
$G$ is a group, and $N$ is a finitely generated normal subgroup of $G$.
We define $\Fi_{A^G}(N) = \Fi(N) = \{ x \in A^G : nx = x \text{ for all } n \in N\}$
and claim that $\Fi(N)$ is a SFT over $G$.
To see this,
let $\Lambda = \{a_1, \ldots, a_m\}$ be a symmetric
generating set for $N$.  Then $\Fi(N)$ is
the SFT determined by the forbidden patterns
\[
\{p : \{1,a_i\} \rightarrow A : a_i \in \Lambda, p(a_i) \neq p(1) \}.
\]
Now suppose $N$ is also of finite index in $G$.  Let $T = \{t_1,
\ldots, t_n\}$ be a complete set of distinct left coset representatives
for $N$ in $G$ with $t_1 = 1$.  We use $T$ as alphabet and define the
\emph{$N$-locked shift} (over $G$, on $T$) to be $L = \Fi_{T^G}(N) \cap S$,
where $S$ is the SFT defined by forbidden patterns
\[
\{ p : \{1, t\} \rightarrow T: t \in T\setminus\{1\}, p(1) = p(t)\}.
\]

\begin{proposition}
$L$ is a nonempty shift of finite type.
Moreover, for any $x \in L$, $gx = x$ if and only if $g \in N$.
\end{proposition}
\begin{proof}
First, we show $L$ is nonempty.
Let $y \in T^G$ be the configuration sending each $g \in G$ to its coset
representative: $y(tn) = t$ whenever $t \in T$, $n \in N$.
Then if $n, n' \in N$, $t \in T$,
\[
ny(tn') = y(n^{-1} t n') = y(t (t^{-1} n^{-1} t) n') = t = y(tn'),
\]
so $ny = y$ and we conclude $y \in \Fi(N)$.  Moreover, $y \in S$,
for if $g = tn \in G$,
\[ gy(1) = ty(1) = y(t^{-1}) \neq y(t^{-1} t_2) = gy(t_2) \]
whenever $t_2 \in T \setminus \{1\}$.

For the second assertion, suppose $x \in L$ and  $g = tn \in G$
satisfies $gx = x$.  Then
\[ tx = x, \quad \text{so} \quad x(1) = tx(t) = x(t). \]
By the restrictions on $S$, $t = 1$ and $g \in N$.
\end{proof}

\section{Commensurability and periodicity}
Using the basic constructions above, we show the properties
of weak and strong periodicity are preserved under finite index extensions
and, in the case when the group is finitely generated, preserved in
finite index subgroups.
As a corollary, we find that weak and strong periodicity are
commensurability invariants.

\begin{proposition}\label{virtual}
Let $G$ be a group.
\begin{enumerate}
\item \label{vrp} If $G$ is virtually weakly periodic, then $G$ is weakly periodic.
\item \label{vsp} If $G$ is virtually strongly periodic, then $G$ is strongly periodic.
\end{enumerate}
\end{proposition}

\begin{proof}
Suppose $H$ is a weakly periodic subgroup of $G$ with $[G : H] = n$ and
$S \subset A^G$ is a nonempty shift of finite type over $G$.  Let $P'$ be a finite set
of forbidden patterns for $S$; by extending patterns if necessary, we may assume
$P' \subset A^\Omega$ for some fixed finite subset $\Omega \subset G$ with $1 \in \Omega$.
Let $T' = \{a_1,
\ldots, a_n\}$ be a set of distinct right coset representatives for $H$ in $G$,
where $a_1 = 1$,
and set $T = T' \Omega$.  Define the higher block map
$\S: A^G \rightarrow I \subset B^H$ as in Section \ref{HigherBlock}, and
let $J \subset B^H$ be the shift of finite type defined by
the following set of forbidden patterns:
\begin{align*}
\{ p :  \{1_H\} \rightarrow B : &\text{there is } p' \in P', t \in T'
\text{ such that }\\
&(p(1_H))(t \omega) = p'(\omega) \text{ for all }
\omega \in \Omega \}.
\end{align*}
Then $I \cap J$ is a shift of finite type and $S = \S^{-1}(I \cap J)$.
Since $S$ is nonempty, $I \cap J$ is nonempty and contains a periodic
configuration $z \in B^H$ by hypothesis; 
i.e. $\St_H(z)$ is nontrivial.
Let $x = \S^{-1}(z) \in S$.
By the remark preceding Section \ref{s:prod}, $\St_G(x)$ is nontrivial, proving assertion (\ref{vrp}).
If in addition $z$ can be chosen so that $\St_H(z)$ is of finite index in $H$,
then $\St_G(x)$ is of finite index in $G$, giving assertion (\ref{vsp}).

\end{proof}

\begin{proposition} \label{hereditary}
Let $G$ be a finitely generated group and $H$ a finite index subgroup of $G$.
\begin{enumerate}
\item If $G$ is weakly periodic, then $H$ is weakly periodic.
\item If $G$ is strongly periodic, then $H$ is strongly periodic.
\end{enumerate}
\end{proposition}

\begin{proof}
Since $H$ contains a finite index subgroup that is normal in $G$,
by Proposition \ref{virtual} we may assume without loss of generality that
$H$ is normal in $G$.
Let $T$ be a complete set of left coset representatives for $H$ in $G$ with
$1 \in T$.

Suppose $A$ is a finite alphabet and $S \subset A^H$ is a SFT
over $H$ defined by forbidden patterns $P \subset A^\Omega$,
where $\Omega \subset H$ is finite.  By regarding $\Omega$ as a subset of $G$,
we can consider the SFT $S' \subset A^G$ with the same forbidden pattern set $P$.
Notice that $S'$ is nonempty, for we can choose $x \in S$
and define $x' \in S'$ by letting $x'(th) = x(h)$ for $t \in T$, $h \in H$.
Notice also that if $y \in S'$, $y|_H \in S$.

Define $L$ to be the $H$-locked shift as in section \ref{LockedShift}.
Then $S' \times L$ is a SFT over $G$.  Moreover,
whenever $y \in S' \times L$, $\St_G(y) \subset H$, for $gy = y$
implies in particular that $gy_2 = y_2$.  Regarding $x = (y_1)|_H$
as a configuration in $S \subset A^H$, it follows that $\St_H(x)$ contains
$\St_G(y)$.  In conclusion, $H$ is weakly (strongly) periodic whenever $G$ is.
\end{proof}

\begin{theorem}
Let $G_1$ and $G_2$ be finitely generated
commensurable groups.  If $G_1$ is weakly (strongly) periodic,
then $G_2$ is weakly (strongly) periodic.
\end{theorem}

\begin{remark}
Emmanuel Jeandel has pointed out to us that the proof of Proposition \ref{hereditary} can be extended
to show that strong periodicity is in fact a hereditary property for all finitely generated
subgroups.  That is, if $G$ is a strongly periodic group and $H$ is a finitely generated
subgroup of $G$, possibly of infinite index, then $H$ is strongly periodic.
However, the property of having a strongly aperiodic SFT is not a hereditary property, as
evidenced by $\Z \leq \Z^2$.
\end{remark}

\subsection{SFTs and quotient groups}

Suppose $G, Q$ are groups and $f: G \rightarrow Q$ is a surjective homomorphism.
Let $N$ be the kernel of $f$ so that $Q = G/N$.   Notice that $f$ induces a bijection
$F: A^Q \rightarrow \Fi_{A^G}(N)$ defined by $F(x) = x \circ f$.  Moreover,
$F$ is $G/N$-equivariant in the following sense: if $q = gN \in Q$
and $x \in A^Q$, then $F(qx) = gF(x)$.  (This correspondence is explained in \cite[Section 1.3]{CAGroups}.)

Now, if $\ov{S} \subset A^Q$
is a SFT defined by forbidden patterns $\ov{P} \subset A^{\Oom}$, $\Oom \subset Q$ finite,
we can choose a (finite) set $\Om \subset G$ such that $f$ maps $\Om$ onto $\Oom$ bijectively.
$f|_\Om$ induces a bijection $g: A^{\Oom} \rightarrow A^{\Om}$. We let $P = g(\ov{P})$ and
$S \subset A^G$ be the SFT with forbidden patterns $P$.

\begin{proposition}\label{t:quotient-strongly}
$F(\ov{S}) = S \cap \Fi (N)$.
\end{proposition}
\begin{proof}
Suppose $x \in A^Q$.
If $\ov{p} \in \ov{P}$ and $p = g(\ov{p}) \in P$, then $\ov{p}$ appears in $x$
if and only if $p$ appears in $F(x)$.  Since $F$ is a bijection from $A^Q$ to $\Fi(N)$,
the result follows.
\end{proof}
It follows from Section \ref{LockedShift} that
$F(\ov{S})$ is a SFT in $A^G$ whenever $N$ is finitely generated.
This leads to a useful general result.
\begin{theorem}
Suppose $\displaystyle{1 \rightarrow N \rightarrow G \rightarrow Q \rightarrow 1}$ is
a short exact sequence of groups and $N$ is finitely generated.  If $G$ is strongly
periodic, then $Q$ is strongly periodic.
\end{theorem}
\begin{proof}

Suppose $\ov{S} \subset A^Q$ is a SFT over $Q$.  Defining $F$ as above, by hypothesis
there is a configuration $F(x) \in F(\ov{S}) \subset A^G$ such that $H = \St_G(F(x))$
is of finite index in $G$.  Of course, $N \subset H$, so $H/N$ is of finite index in
$G/N = Q$.  Moreover, by equivariance of $F$, $qx = x$ for every $q \in H/N$.
In conclusion, $x$ has finite index stabilizer in $Q$.
\end{proof}

As an example application,
this result places restrictions on the SFTs which can be defined on groups of polynomial growth.

\begin{cor}\label{c:polynomial}
If $G$ is a finitely generated group of polynomial growth $\gamma(n) \sim n^d$, where $d \geq 2$,
then $G$ is not strongly periodic.
\end{cor}
\begin{proof}
By Gromov's theorem, $G$ is virtually nilpotent, so by Proposition \ref{hereditary},
we may assume $G$ is nilpotent. Since $G$ is not virtually cyclic,
  there exists a surjective homomorphism
$f: G \rightarrow \Z^2$ (\cite[Lemma 13]{GeodesicGrowthPolynomial}) and as every subgroup of a finitely generated
nilpotent group is itself finitely generated, $N = \text{ker } f$ is finitely
generated.  Since $\Z^2$ is not strongly periodic, $G$ is not strongly periodic.
\end{proof}

Corollary \ref{c:polynomial} does not rule out the existence of periodic points for SFTs defined on groups of polynomial growth. However, the examples we know of  groups
with strongly aperiodic SFTs
 (such as $\mathbb{Z}^n$ and the discrete Heiseinberg group) are all groups of polynomial growth. This leads to the following question.

\begin{question}
Is there a group of non-linear polynomial growth on which every shift of finite type has a periodic point? 
\end{question}

We also make the following conjecture.

\begin{conjecture}
A group is strongly periodic if and only if it is virtually cyclic.
\end{conjecture}

\bibliographystyle{plain}

\end{document}